\documentclass[11pt,reqno]{amsart}

\usepackage{amsmath,amssymb,amsthm,extarrows}
\usepackage{graphicx}
\usepackage[dvipsnames]{xcolor}

\usepackage[cmtip,all]{xy}
\usepackage[hidelinks,colorlinks,linkcolor=red,citecolor=ForestGreen,urlcolor=blue]{hyperref}

\theoremstyle{plain}
\newtheorem{theorem}{Theorem}[section]
\newtheorem{lemma}[theorem]{Lemma}
\newtheorem{proposition}[theorem]{Proposition}
\newtheorem{corollary}[theorem]{Corollary}

\theoremstyle{definition}

\theoremstyle{remark}
\newtheorem{remark}[theorem]{Remark}

\numberwithin{equation}{section}

\begin{document}

\title[Derived equivalent but not tilting-cotilting equivalent]{Derived equivalent gentle algebras which are not tilting-cotilting equivalent}
\author{Zhiyang Ni}
\address{University of Science and Technology of China, Hefei 230026, Anhui, PR China}
\email{zhiyangni@mail.ustc.edu.cn}


\date{\today}

\date{\today}

\begin{abstract}
    We give examples of two derived equivalent gentle algebras which are not tilting-cotilting equivalent.
\end{abstract}

\maketitle
\tableofcontents

\section{Introduction}
\par Derived equivalence is one of the fundamental equivalence relations in the representation theory of finite-dimensional algebras. According to \cite{Hap}, an abundant supply of derived equivalences is provided by tilting-cotilting equivalences. It is therefore natural to ask: to what extent can derived equivalence be realized by tilting-cotilting equivalence? We also care about for which class of algebras these two notions of equivalences coincide.
\par It is advisable to investigate this problem first on gentle algebras, a class of algebras whose representation theory and derived classification are well-studied \cite{APS}. As shown in \cite{Bob}, for gentle algebras whose Gabriel quivers contain at most two cycles, derived equivalence turns out to be the same as tilting-cotilting equivalence. One might therefore suspect this for all gentle algebras. This paper proves that this expectation is false by constructing an explicit family of examples and analyzing their derived and tilting-cotilting behaviour.
\par Our main results are twofold. First, we introduce a family of gentle algebras $\Lambda(L,R)$ defined by two total orderings of a finite set $E=\{1,2,\cdots,n\}$, and show that for the special case $L=R=(1,2,\cdots,n)$, the algebra $\Lambda=\Lambda(L,L)$ is not affected by tilting-cotilting operations:
\begin{theorem}\label{thm A}
    The algebra $\Lambda$ is tilting-cotilting unique for $n\geq1$.
\end{theorem}
\par The proof proceeds in three steps. On one hand, we compute the Cartan matrix of $\Lambda$ and show that it is congruent to $I_n$ modulo $2$. Second, using mutation invariance and string combinatorics, we show that the dimension vectors of the indecomposable summands of a tilting $\Lambda$-module $T$ are congruent to the standard basis of $\mathbb{Z}^n$ modulo $2$. the dimension vectors of the summands are congruent to the standard basis vectors modulo 2. Finally, we pass to trivial extensions \cite{Ric} and use the derived equivalence classification of Brauer graph algebras \cite{OZ} to show that every algebra obtained from $\Lambda$ by a tilting or cotilting module must again be of the form $\Lambda(U,V)$. The theorem is proved by combining all these clues together.
\par On the other hand, we will show that $\Lambda$ is far from being derived unique:
\begin{theorem}\label{thm B}
    For every $n\geq5$, there is an algebra $\Lambda'$ derived equivalent to $\Lambda$ but not tilting-cotilting equivalent to it.
\end{theorem}
\par The proof uses the geometric model and complete derived invariants introduced  by Amiot, Plamondon and Schroll in \cite{APS}. After giving the construction of $\Lambda'$, we compute their numerical invariants of \cite[Theorem 7.4]{APS} explicitly and show that they coincide.
\par The lower bound $n\geq5$ in Theorem \ref{thm B} is sharp. In fact, for $n\leq4$ we will show that the algebra $\Lambda$ is derived unique.
\par The paper is organized as follows. In Section 2 we introduce the family $\Lambda(L,R)$ and record some elementary properties of the algebra $\Lambda=\Lambda(L,L)$. In Section 3 we prove Theorem \ref{thm A} by passing to trivial extensions and using the derived equivalence classification of Brauer graph algebras together with a calculation of Cartan matrices. In Section 4 we construct $\Lambda'$, compute the surface invariants associated with $\Lambda$ and $\Lambda'$, and prove Theorem \ref{thm B}. We conclude by discussing the exceptional cases $n\leq4$.

\subsection{Notations and conventions}
\par Our notations and conventions are basically consistent with those of \cite{ASS1}. In particular, we assume that all modules are on the right, tilting and cotilting modules are in the classical sense \cite[Definition VI.2.1]{ASS1}, and all tilting and cotilting modules are basic. Arrows of quivers are composed from left to right, but morphisms are from right to left. Throughout this paper, let $\Bbbk$ be a fixed algebraically closed field.

\section{The family \texorpdfstring{$\Lambda(L,R)$}{Λ(L,R)}}
\par In this section, we construct a family of gentle algebras and describe their properties. 
\par Fix an integer $n\geq1$. Let $E=\{1,2,\cdots,n\}$, and let $L=(a_1,\cdots,a_n)$ and $R=(b_1,\cdots,b_n)$ be two total orderings of $E$. We define a bound quiver algebra $\Lambda(L,R)=\Bbbk Q/I$ as follows: the quiver $Q$ is formed by vertices $1,2,\cdots,n$ and arrows $\alpha_s:a_s\to a_{s+1}$ and $\beta_s:b_s\to b_{s+1}$ ($1\leq s\leq n-1$); the admissible ideal $I$ is generated by all composable paths of the form $\alpha_i\beta_j$ or $\beta_j\alpha_i$. We denote the primitive idempotents of $\Lambda(L,R)$ by $e_1,\cdots,e_n$. From this construction, one easily sees that $\Lambda(L,R)$ is a finite-dimensional connected gentle algebra.
\par Without loss of generality, we will take $L=(1,2,\cdots,n)$, the natural linear ordering of $E$, in what follows. We also write $\Lambda=\Lambda(L,L)$ for convenience. Therefore, we have the following explicit presentation of $\Lambda$:
\[\Lambda=\Bbbk\Bigg(\xymatrix@!{
    1\ar@<1mm>[r]^-{\alpha_1}\ar@<-1mm>[r]_-{\beta_1}
    &2\ar@<1mm>[r]^-{\alpha_2}\ar@<-1mm>[r]_-{\beta_2}
    &\cdots\ar@<1mm>[r]^-{\alpha_{n-1}}\ar@<-1mm>[r]_-{\beta_{n-1}}
    &n
}\Bigg)/\langle\alpha_i\beta_{i+1},\beta_i\alpha_{i+1}\mid 1\leq i\leq n-2\rangle.\]
It then follows immediately that $\Lambda$ is self-dual, i.e., isomorphic to its opposite algebra. Moreover, $\Lambda$ is of finite global dimension, since it contains no full-relation oriented cycle \cite[Proposition 1.25]{OPS}.

\section{Tilting-cotilting uniqueness}
\par In this section we prove the following theorem.
\begin{theorem}\label{thm:tc-unique}
    The algebra $\Lambda$ is tilting-cotilting unique for every integer $n\geq1$.
\end{theorem}
\par Here, by ``tilting-cotilting unique'' we mean that for any tilting module or cotilting module $T$ over $\Lambda$, there is an algebra isomorphism $\mathop{\mathrm{End}}_\Lambda(T)\cong\Lambda$. Since $\Lambda\cong\Lambda^{\mathrm{op}}$, it suffices to prove only the case where $T$ is tilting.

\subsection{Cartan matrices}
The proof involves two different ways of calculating Cartan matrices of algebras of the form $\Lambda(L,R)$. One is by direct computation based on the definition, as the following lemma illustrates.
\begin{lemma}\label{lem:Cartan matrix}
    For the algebra $\Lambda(L,R)$, its Cartan matrix $\mathbf{C}_{\Lambda(L,R)}=(c_{ij})_{n\times n}$ is given by $c_{ii}=1$ for $1\leq i\leq n$ and
    \[c_{ji}=\left\{\begin{array}{ll}
        i<_Lj, & 1; \\
        i\geq_Lj, & 0
    \end{array}\right\}+\left\{\begin{array}{ll}
        i<_Rj, & 1; \\
        i\geq_Rj, & 0
    \end{array}\right\}\quad\text{for $i\ne j$.}\]
    In particular,
    \[\mathbf{C}_\Lambda=\begin{bmatrix}
        1&0&0&\cdots&0&0\\
        2&1&0&\cdots&0&0\\
        2&2&1&\cdots&0&0\\
        \vdots&\vdots&\vdots&\ddots&\vdots&\vdots\\
        2&2&2&\cdots&1&0\\
        2&2&2&\cdots&2&1\\
    \end{bmatrix}.\]
\end{lemma}
\begin{proof}
    By definition, $c_{ji}=\dim_\Bbbk e_i\Lambda e_j=\dim_\Bbbk\mathop{\mathrm{Hom}}_\Lambda(e_j\Lambda,e_i\Lambda)$. In our case, it equals the number of non-zero paths from $i$ to $j$ in $\Lambda$. If $i=j$, there is only one trivial path. If $i\ne j$, any non-zero non-trivial path must be formed entirely by arrows $\alpha_1,\cdots,\alpha_{n-1}$, or entirely by arrows $\beta_1,\cdots,\beta_{n-1}$. The first case is equivalent to $i<_Lj$, and the second case is equivalent to $i<_Rj$. So our formula holds. 
\end{proof}
\begin{corollary}\label{lem:same Cartan => isomorphic}
    If $\Lambda$ and $\Lambda(U,V)$ share the same Cartan matrix modulo 2, they are isomorphic.
\end{corollary}
\begin{proof}
    By the above lemma, for $i\ne j$ in $E$ we have $i<_Uj$ if and only if $i<_Vj$. This implies that the two orderings $U$ and $V$ are identical. After relabelling, we clearly have $\Lambda(U,U)\cong\Lambda(L,L)$, hence $\Lambda\cong\Lambda(U,V)$.
\end{proof}

\subsection{Dimension vectors of tilting modules}
\par The second calculation is subtler. As we will show, it explores a certain symmetry of the algebra $\Lambda$ and its tilting modules.
\par Define an algebra endomorphism $\sigma:\Lambda\to\Lambda$ by the condition that $\sigma$ fixes every primitive idempotent $e_i$ ($1\leq i\leq n$) and swaps every pair of arrows $\alpha_i,\beta_i$ for $1\leq i\leq n-1$. It is easy to see that $\sigma$ is an algebra automorphism. For any $\Lambda$-module $M$, denote by $M^\sigma$ the module $M$ equipped with $\sigma$-twisted right $\Lambda$-action. We denote the elements of $M^\sigma$ by $m^\sigma$; then $m^\sigma x=m\sigma(x)$ by definition.
\begin{lemma}
    Every indecomposable summand $M$ of a tilting module $T$ is $\sigma$-invariant, i.e., $M\cong M^\sigma$ as $\Lambda$-modules.
\end{lemma}
\begin{proof}
    Every tilting module is a support $\tau$-tilting module \cite[Proposition 2.2(b)]{AIR}. Since $\Lambda$ is gentle, its support $\tau$-tilting graph $\mathcal{H}(\mathsf{s\tau\text{-}tilt}\Lambda)$ is connected \cite{FGLZ}. Therefore, there is a sequence of mutations from the regular module $\Lambda$ to $T$ in $\mathcal{H}(\mathsf{s\tau\text{-}tilt}\Lambda)$, which may be identified by \cite[Theorem 3.2, Corollary 3.9]{AIR} with a sequence of mutations from $\Lambda[0]$ to $P^\bullet_T$, the minimal projective resolution of $T[0]$, in $\mathcal{H}(\mathsf{2\text{-}silt}\Lambda)$, the 2-term silting quiver. Clearly $\sigma$ induces a triangulated automorphism of $\mathbf{D}^b(\Lambda\text{-}\mathrm{mod})$, which is also denoted by $\sigma$, and every indecomposable summand of $\Lambda[0]$ is $\sigma$-invariant. Thus, by \cite[Proposition 2.1]{Dug} and induction on the length of the mutation sequence, we see that every indecomposable summand of $P^\bullet_T$ is also $\sigma$-invariant. Taking homology at degree 0, we see that every indecomposable summand $M$ of $T$ is $\sigma$-invariant.
\end{proof}
\par Denote by $\mathbb{F}_2$ the finite field with two elements $0,1$. Denote by $\varepsilon_1,\cdots,\varepsilon_n$ the standard basis of the $n$-dimensional vector space $\mathbb{F}_2^n$. The dimension vector of a module $M$ is denoted by $\underline{\dim}_\Bbbk(M)$.
\begin{lemma}\label{lem:dim-vec of tilting modules form std basis}
    Let $T$ be a tilting $\Lambda$-module, and let $T_1,\cdots,T_n$ be all of its indecomposable summands. Then, up to permutation, we have
    \[\underline{\dim}_\Bbbk(T_i)\equiv\varepsilon_i\mod2.\]
\end{lemma}
\begin{proof}
    By \cite[Section 3]{BR}, indecomposable modules over $\Lambda$ are in one-one correspondence with equivalence classes of strings and bands over $\Lambda$. Since band modules admit non-zero self-extensions \cite[Page 165]{BR} but $T$ is tilting, the $T_i$ are not band modules. Let $u_i$ be the string defining $T_i$, so $T_i\cong M(u_i)$. From $T_i\cong T_i^\sigma$ we obtain $M(u_i)\cong M(u_i)^\sigma=M(\sigma(u_i))$, so $\sigma(u_i)=u_i$ or $u_i^{-1}$. In the first case, $u_i$ is a trivial string since no arrow is $\sigma$-invariant. In the second case, $u_i$ is of even length, and its second half is determined by the first half, hence of the form $p\sigma(p)^{-1}$ for some string $p$. In both cases, it is true that $\underline{\dim}_\Bbbk(T_i)\equiv\varepsilon_j\mod2$ for some $j$, because a vertex in the first half of $u_i$ along with the corresponding vertex in the second half together contribute twice, but the middle vertex, say $j$, contributes only once. It remains to show that the correspondence $i\mapsto j$ is bijective.
    \par The fact that $\underline{\dim}_\Bbbk(T_i)$ ($i=1,\cdots,n$) form a $\mathbb{Z}$-basis of the Grothendieck group $K_0(\Lambda)\cong\mathbb{Z}^n$ \cite[proof of Proposition 4.5]{ASS1} implies that $\underline{\dim}_\Bbbk(T_i)\mod2$ ($i=1,\cdots,n$) form a $\mathbb{F}_2$-basis of $\mathbb{F}_2^n$. Therefore the indices $j$ occurring above are all distinct, and after reindexing we obtain $\underline{\dim}_\Bbbk(T_i)\equiv\varepsilon_i\mod2$.
\end{proof}
\begin{corollary}\label{cor:Cartan matrix of tilting is I mod2}
    If $\Lambda'\cong\mathop{\mathrm{End}}_\Lambda(T)$ for some tilting module $T$ over $\Lambda$, then the Cartan matrix $\mathbf{C}_{\Lambda'}\equiv I_n\mod2$.
\end{corollary}
\begin{proof}
    Recall a formula: if $\langle-,-\rangle$ is the Euler characteristic of $\Lambda$ \cite[Definition III.3.11]{ASS1}, then $(\mathbf{C}_{\Lambda'})_{ji}=\langle\underline{\dim}_\Bbbk(T_j),\underline{\dim}_\Bbbk(T_i)\rangle$. Indeed, the indecomposable projective summands of $\Lambda'$ are $\mathop{\mathrm{Hom}}_\Lambda(T,T_i)$ ($i=1,\cdots,n$), so
    \begin{align}
        (\mathbf{C}_{\Lambda'})_{ji}
        &=\dim_\Bbbk\mathop{\mathrm{Hom}}\nolimits_{\Lambda'}(\mathop{\mathrm{Hom}}\nolimits_\Lambda(T,T_j),\mathop{\mathrm{Hom}}\nolimits_\Lambda(T,T_i))
        \tag{\cite[Definition 3.7, Lemma I.4.2]{ASS1}}\\
        &=\dim_\Bbbk\mathop{\mathrm{Hom}}\nolimits_\Lambda(T_j,T_i)
        \tag{\cite[Lemma VI.3.2]{ASS1}}\\
        &=\sum_{r=0}^\infty(-1)^r\dim_\Bbbk\mathop{\mathrm{Ext}}\nolimits_\Lambda^r(T_j,T_i)
        \tag{$T$ is tilting}\\
        &=\langle\underline{\dim}_\Bbbk(T_j),\underline{\dim}_\Bbbk(T_i)\rangle.
        \tag{\cite[Proposition III.3.13]{ASS1}}
    \end{align}
    Since the matrix of $\langle-,-\rangle$ is $\mathbf{C}_\Lambda^{-\mathrm{T}}$ and $\mathbf{C}_\Lambda^{-\mathrm{T}}\equiv I_n\mod2$ by Lemma \ref{lem:Cartan matrix}, we see that $(\mathbf{C}_{\Lambda'})_{ji}\equiv\langle\varepsilon_j,\varepsilon_i\rangle\equiv\delta_{i,j}\mod2$ by Lemma \ref{lem:dim-vec of tilting modules form std basis}, and so $\mathbf{C}_{\Lambda'}\equiv I_n\mod2$ as claimed. 
\end{proof}

\subsection{Proof of Theorem \ref{thm:tc-unique}}
\par Now we can prove Theorem \ref{thm:tc-unique}. The idea is roughly as follows: pass to their trivial extensions, which are Brauer graph algebras; show that $\Lambda'=\mathop{\mathrm{End}}_\Lambda(T)$ is also of the form $\Lambda(U,V)$; apply the above lemmas to show that $\Lambda\cong\Lambda'$.
\par Before the proof, we note that when $n=1$ the algebra $\Lambda$ is isomorphic to $\Bbbk$ and the conclusion is trivial. Since the classification theorem \cite[Theorem 7.12]{OZ} excludes local algebras, we assume $n\geq2$ in the following proof.
\begin{proof}[Proof of Theorem \ref{thm:tc-unique}]
    Let $\Lambda'=\mathop{\mathrm{End}}_\Lambda(T)$, where $T$ is a tilting $\Lambda$-module. $\Lambda'$ is gentle by \cite[Theorem 8.1]{APS}. Denote $B=\mathop{\mathrm{Triv}}(\Lambda)$, the trivial extension of $\Lambda$; similarly denote $B'$ for $\Lambda'$. They are both Brauer graph algebras by \cite[Theorem 1.2]{Sch}, so they are defined by Brauer graphs $\Gamma$ and $\Gamma'$, respectively. Moreover, $\Lambda$ and $\Lambda'$ are obtained from them respectively by admissible cuts along the arrows closing the maximal permitted paths; see \cite[Section 4]{Sch} for relevant definitions and theorems. A moment's calculation shows that $\Gamma$ has the following properties: it has 2 vertices both of multiplicity 1, and $n$ edges connecting them, thus bipartite.
    \par Since $\Lambda$ and $\Lambda'$ are derived equivalent, so are $B$ and $B'$ \cite[Theorem 3.1]{Ric}. By the criterion for derived equivalence between Brauer graph algebras \cite[Theorem 7.12]{OZ}, we see that $\Gamma'$ shares the same properties as $\Gamma$ as above. More precisely, $\Gamma'$ is bipartite, so it has no loops and thus each of its $n$ edges joins the two distinct vertices. Therefore, $\Gamma'$ defines two cyclic orderings of $E=\{1,\cdots,n\}$, and they turn into two total orderings $U,V$ of $E$ by an admissible cut yielding $\Lambda'$, thus $\Lambda'\cong\Lambda(U,V)$.
    \par Finally, since $\mathbf{C}_{\Lambda'}\equiv I_n\mod2$ by Corollary \ref{cor:Cartan matrix of tilting is I mod2}, the result now follows from Lemma \ref{lem:same Cartan => isomorphic} as we have just shown that $\Lambda'$ is of the form $\Lambda(U,V)$.
\end{proof}
\begin{remark}\label{rmk}
    The first two paragraphs of the proof of Theorem \ref{thm:tc-unique} imply that the class of gentle algebras of the form $\Lambda(L,R)$ is closed under derived equivalence.
\end{remark}

\section{Non-derived uniqueness}
\par For $n\geq5$ we take $\Lambda'=\Lambda(L,R)$ where
\[R=(5,6,\cdots,n,1,2,3,4).\]
We are going to prove the following theorem.
\begin{theorem}\label{thm:derived equivent}
    For $n\geq5$, the algebras $\Lambda$ and $\Lambda'$ are derived equivalent but not isomorphic. In particular, $\Lambda$ is not derived unique.
\end{theorem}
\par That $\Lambda\not\cong\Lambda'$ is easy to see by comparing their Cartan matrices. Alternatively, notice that the projective $\Lambda$-module at vertex $n$ is simple, whereas $\Lambda'$ has no simple projective module.
\par To check that they are derived equivalent, we will compare the numerical invariants introduced in \cite[Theorem 7.4]{APS} on the dissected surfaces $(S,M,P,\Delta)$ and $(S',M',P',\Delta')$ associated with them, respectively. \cite[Theorem 7.4]{APS} is divided into three parts. We check them one by one.

\subsection{The ribbon surface}
\par We first describe the dissected surfaces.
\par In general, the algebra $\Lambda(U,V)$ has precisely two maximal permitted paths: the path consisting entirely of $\alpha$-arrows and the path consisting entirely of $\beta$-arrows. Consequently, its ribbon graph has two vertices, which we denote by $v_1$ and $v_2$, and $n$ edges $e_1,\ldots,e_n$ connecting them. The cyclic ordering at $v_1$ is induced by $U$, and that at $v_2$ is induced by $V$.
\par In our case, $L$ and $R$ induce the same cyclic ordering $1<2<\cdots<n<1$. Hence $\Lambda$ and $\Lambda'$ have the same underlying ribbon graph and therefore we shall identify $S$ and $S'$. But the positions of marked points are different. More precisely, the two $\color{green}\circ$-marked points on $S$ both lie between $1$ and $n$, but on $S'$, one is between $1$ and $n$, the other is between $4$ and $5$.

\begin{lemma}\label{lem:ribbon surface}
    With notation as in \cite[Theorem 7.4]{APS}, we have
    \[\begin{gathered}
        g=g'=\Bigg\{\!\!\!\begin{array}{ll}
            \frac{n-1}{2}, & n\text{ odd}; \\
            \frac{n}{2}-1, & n\text{ even},
        \end{array}\quad
        b=b'=\bigg\{\!\!\!\begin{array}{ll}
            1, & n\text{ odd}; \\
            2, & n\text{ even},
        \end{array}\\
        \#M=\#M'=4,\quad\#P=\#P'=0.
    \end{gathered}\]
    Moreover, if $n$ is odd, the two $\color{green}\circ$-marked points lie on the unique boundary component, while if $n$ is even, the two $\color{green}\circ$-marked points lie on distinct boundary components, both for $\Lambda$ and for $\Lambda'$.
\end{lemma}
\begin{proof}
    The numbers of $\color{green}\circ$-marked points and $\color{red}\bullet$-marked points both equal $2$. Since there is no full-relation oriented cycle in $\Lambda$ and $\Lambda'$, there is no puncture on $S$ and $S'$, so $P=P'=\varnothing$.
    \par The number of boundary components equals the number of faces \cite[Definition 1.4]{OPS} on the ribbon graph $\Gamma_{\Lambda}$ associated with $\Lambda$. A direct counting yields the formula for $b$ and $b'$. Finally, employing the formula
    \[n=\#M_{\color{green}\circ}+\#P+b+2g-2,\]
    in \cite[Proposition 3.12]{APS}, one finds the genera $g$ and $g'$ as stated.
    \par The positions of $\color{green}\circ$-marked points are easy to see.
\end{proof}

\subsection{Boundary winding numbers}
\par When $n$ is odd, take a simple closed curve $c$ on $S$ which is homotopic to the unique boundary component, with orientation induced from $S$. When $n$ is even, take two simple closed curves $c_1,c_2$ on $S$ which are homotopic to the two boundary components, respectively, with induced orientation.
\begin{lemma}\label{lem:boundary winding number}
    The boundary winding numbers for $\Lambda$ and $\Lambda'$ coincide. More precisely,
    \par(1) When $n$ is odd, $w^{\Delta^*}(c)=4-2n$.
    \par(2) When $n$ is even, $w^{\Delta^*}(c_1)=w^{\Delta^*}(c_2)=2-n$.
\end{lemma}
\begin{proof}
    \par(1) When $n$ is odd, there is only one boundary component of $S$, on which there are two $\color{red}\bullet$-marked points, and there are $n$ $\color{red}\bullet$-arcs connecting one to the other. So the winding number is $4-2n$ by \cite[Lemma 3.18]{APS}.
    \par(2) When $n$ is even, there are two boundary components of $S$, on each of which there is one $\color{red}\bullet$-marked point, and there are $n$ $\color{red}\bullet$-arcs connecting one to the other. So the winding numbers are both $2-n$, by the same counting strategy as above.
    \par The same calculation holds for $\Lambda'$.
\end{proof}

\subsection{A convenient family of closed curves}\label{subsec:a familt of closed curves}
\par Having compared the topological data and boundary winding numbers of two marked surfaces, we continue with \cite[Theorem 7.4(3)]{APS}. Since $g=g'\geq2$, we are in case (b) of \cite[loc. cit.]{APS}.
\par By embedding the ribbon graph $\Gamma_\Lambda$ into $S$ as a deformation retract, we can view the edges $e_1,\cdots,e_n$ as $\color{green}\circ$-arcs on $S$. We now introduce a convenient collection of simple closed curves on $S$. For $i\ne j$, let $\gamma_{ij}$ be the simple closed curve on $S$ based at $v_1$, formed by first going along $e_i$ and returning along $e_j$. It is clear that the $\gamma_{ij}$ generate the fundamental group $\pi_1(S)$, and using \cite[Lemma 3.18]{APS} we have
\begin{equation}\label{eq:winding number}
    \begin{aligned}
        w^{\Delta^*}(\gamma_{ij})&=0,\quad\forall i,j,\\
        w^{\Delta'^*}(\gamma_{ij})&=\Bigg\{\!\!
        \begin{array}{ll}
            0, & \text{if $L,R$ give the same order to $i,j$}; \\
            \pm2, & \text{if $L,R$ give opposite orders to $i,j$}.
        \end{array}
    \end{aligned}
\end{equation}
\begin{lemma}
    For every simple closed curve $\gamma$ on $S$, the winding numbers $w^{\Delta^*}(\gamma)$ and $w^{\Delta'^*}(\gamma)$ are even.
\end{lemma}
\begin{proof}
    We prove the lemma for $w^{\Delta^*}(\gamma)$ only, because for $w^{\Delta'^*}(\gamma)$ the proof is similar. By \cite[Proposition 3.4(4)]{APS}, $w^{\Delta^*}(\gamma)$ factors through an element of the cohomology group $H^1(S;\mathbb{F}_2)$, which is isomorphic to the dual vector space $H_1(S;\mathbb{F}_2)^*$. So there is a linear functional $f:H_1(S;\mathbb{F}_2)\to\mathbb{F}_2$ and a commutative diagram
    \[\xymatrix{
        \pi_1^{\mathrm{free}}(S)\ar[r]^-{w^{\Delta^*}}\ar[d]&\mathbb{Z}\ar[d]^{q}\\
        H_1(S;\mathbb{F}_2)\ar[r]^-{f}&\mathbb{F}_2,
    }\]
    the vertical mappings being canonical. 
    \par To show $\mathop{\mathrm{Im}}w^{\Delta^*}\subseteq2\mathbb{Z}$ it suffices to show that $f=0$ in the above diagram. First, we compute $f([\gamma_{ij}])=q(w^{\Delta^*}(\gamma_{ij}))=q(\text{an even integer})=0$. Then, notice that the homology classes $[\gamma_{ij}]$ span $H_1(S;\mathbb{F}_2)$, so $f=0$ as expected.
\end{proof}

\par By this lemma, we exclude the case (i) of \cite[Theorem 7.4(3)(b)]{APS}. When $n\equiv2\mod4$, the boundary winding numbers are all $0$ by Lemma \ref{lem:boundary winding number}. So we are in case (ii) of \cite[loc. cit.]{APS} and $\Lambda$ and $\Lambda'$ are thus derived equivalent.

\subsection{The quadratic form and the intersection form}\label{subsec:quadratic form & intersection form}
\par Now assume that $n$ is odd or $n\equiv0\mod4$. We are now in case (iii) of \cite[loc. cit.]{APS}, and we must compute the Arf invariants of $\Lambda$ and $\Lambda'$.
\par In our situation, \cite[Lemma 7.5]{APS} is applicable and, to begin with, let us describe the homology group $H_1(S;\mathbb{F}_2)$. It is a $\mathbb{F}_2$-vector space of dimension $n-1$ and we take $[\gamma_{1n}],\cdots,[\gamma_{n-1,n}]$ as its basis. It is canonically isomorphic to the vector space with $\{A\subseteq E\mid\#A\text{ is even}\}$ as the underlying set and symmetric difference as addition; an isomorphism is explicitly given by $[\gamma_{in}]\mapsto\{i,n\}$, $i=1,\cdots,n-1$. We denote the latter vector space by $V$ and we shall use it later as an alternative description of $H_1(S;\mathbb{F}_2)$.
\par Next, we must calculate the quadratic form and the intersection form $(-,-)$ on $H_1(S;\mathbb{F}_2)$ mentioned in \cite[Lemma 7.5]{APS}. Let us do it for $\Lambda$ first. There is a unique quadratic form $q_{\Delta^*}:H_1(S;\mathbb{F}_2)\to\mathbb{F}_2$, such that
\[\begin{aligned}
    &q_{\Delta^*}(x+y)=q_{\Delta^*}(x)+q_{\Delta^*}(y)+(x,y),&&\forall x,y\in H_1(S;\mathbb{F}_2);\\
    &q_{\Delta^*}([\gamma_{ij}])=\frac12w^{\Delta^*}(\gamma_{ij})+1\xlongequal{(\ref{eq:winding number})}1,&&\forall i,j\;(i\ne j).
\end{aligned}\]
For $\Lambda'$, a similar calculation shows that the quadratic form $q_{\Delta'^*}:H_1(S;\mathbb{F}_2)$ $\to\mathbb{F}_2$ satisfies
\[\begin{aligned}
    &q_{\Delta'^*}(x+y)=q_{\Delta'^*}(x)+q_{\Delta'^*}(y)+(x,y),&&\forall x,y\in H_1(S;\mathbb{F}_2);\\
    &\begin{aligned}
        q_{\Delta'^*}([\gamma_{ij}])&=\frac12w^{\Delta'^*}(\gamma_{ij})+1\\
        &=\Bigg\{\!\!
        \begin{array}{ll}
        1, & \text{if $L,R$ give the same order to $i,j$}; \\
        0, & \text{if $L,R$ give opposite orders to $i,j$},
    \end{array}\end{aligned}
    &&\forall i,j\;(i\ne j).
\end{aligned}\]
\begin{lemma}
    \par(1) If $n$ is odd, $(-,-)$ is non-degenerate;
    \par(2) If $n\equiv0\mod4$, $(-,-)$ is degenerate, and its radical is $1$-dimensional, spanned by $[\gamma_{1n}]+\cdots+[\gamma_{n-1,n}]$, or alternatively, spanned by $E\in V$ via the canonical isomorphism.
\end{lemma}
\begin{proof}
    \par For later use, let us write
    \[J_n=\begin{bmatrix}
        0&1&\cdots&1&1\\
        1&0&\cdots&1&1\\
        \vdots&\vdots&\ddots&\vdots&\vdots\\
        1&1&\cdots&0&1\\
        1&1&\cdots&1&0
    \end{bmatrix}\in\mathbb{F}_2^{n\times n}.\]
    \par Let us calculate the matrix of $(-,-)$ under the basis $[\gamma_{1n}],\cdots,[\gamma_{n-1,n}]$. For $\Lambda$ and $1\leq i\ne j<n$ we have
    \[\begin{aligned}
        ([\gamma_{in}],[\gamma_{jn}])
        &=q_{\Delta^*}([\gamma_{in}]+[\gamma_{jn}])+q_{\Delta^*}([\gamma_{in}])+q_{\Delta^*}([\gamma_{jn}])\\
        &=q_{\Delta^*}([\gamma_{ij}])+q_{\Delta^*}([\gamma_{in}])+q_{\Delta^*}([\gamma_{jn}])\\
        &=1+1+1
        =1.
    \end{aligned}\]
    So the matrix of $(-,-)$ is $J_{n-1}$.
    \par By linear algebra, we see that the matrix $J_{n-1}$ is non-degenerate if $n$ is odd, and is degenerate if $n$ is even. In the latter case, its radical is the $1$-dimensional subspace spanned by the vector $(1,1,\cdots,1)^{\mathrm{T}}$. Thus (1) and (2) follow.
\end{proof}
\begin{remark}
    The intersection form $(-,-)$ is an intrinsic topological pairing on $H_1(S;\mathbb{F}_2)$ and is independent of the particular dissection. However, the two quadratic forms $q_{\Delta^*}$ and $q_{\Delta'^*}$ are not equal, as is shown by explicit formulae above.
\end{remark}

\subsection{The Arf invariant}
\par According to \cite[Lemma 7.5]{APS}, the forms $q_{\Delta^*}$ and $q_{\Delta'^*}$ descend to quadratic forms $\bar q_{\Delta^*}$ and $\bar q_{\Delta'^*}$, respectively, on the closed surface $\bar S$ obtained from $S$ by capping the boundary components by discs. The inclusion map $S\to\bar S$ induces a long exact sequence of homology groups:
\[0
    \to\underset{\cong\mathbb{F}_2}{H_2(\bar S;\mathbb{F}_2)}
    \to\underset{\cong\Big\{\!\!\scriptsize\begin{array}{ll}
        \mathbb{F}_2, &n \text{ odd}, \\
        \mathbb{F}_2^2, &n \text{ even}
    \end{array}}{H_2(\bar S,S;\mathbb{F}_2)}
    \to\underset{\cong\mathbb{F}_2^{n-1}}{H_1(S;\mathbb{F}_2)}
    \to\underset{\cong\mathbb{F}_2^{2g}}{H_1(\bar S;\mathbb{F}_2)}
    \to\underset{\cong0}{H_1(\bar S,S;\mathbb{F}_2)}.
\]
We see that: if $n$ is odd then $H_1(S;\mathbb{F}_2)\cong H_1(\bar S;\mathbb{F}_2)$; if $n$ is even then the induced map $H_1(S;\mathbb{F}_2)\to H_1(\bar S;\mathbb{F}_2)$ is surjective and realizes $H_1(\bar S;\mathbb{F}_2)$ as the quotient space of $H_1(S;\mathbb{F}_2)$ by the radical of $(-,-)$. So $(-,-)$ descends to a form $\overline{(-,-)}$ on $H_1(\bar S;\mathbb{F}_2)$.
\par By \cite[Lemma 7.5]{APS}, the Arf invariants of $\Lambda$ and $\Lambda'$ are calculated by first specifying a symplectic geometric basis $\bar u_1,\bar v_1,\cdots,\bar u_g,\bar v_g$ of $H_1(\bar S;\mathbb{F}_2)$, then calculating the sums
    \[\mathop{\mathrm{Arf}}(\Lambda)=\sum_{i=1}^g\bar q_{\Delta^*}(\bar u_i)\bar q_{\Delta^*}(\bar v_i),\quad
    \mathop{\mathrm{Arf}}(\Lambda')=\sum_{i=1}^g\bar q'_{\Delta^*}(\bar u_i)\bar q'_{\Delta^*}(\bar v_i).\]
\par In order to do this, we take vectors
\[\begin{aligned}
    u_i=(0,\cdots,0,\overset{2i-1}{1},\overset{2i}{1},0,\cdots,\overset{n-1}{0}),\\
    v_i=(1,\cdots,1,\overset{2i-1}{1},\overset{2i}{0},0,\cdots,\overset{n-1}{0}),
\end{aligned}\qquad(i=1,\cdots,g)\]
with respect to the basis $[\gamma_{1n}],\cdots,[\gamma_{n-1,n}]$, and denote by $\bar u_1,\bar v_1,\cdots,\bar u_g,\bar v_g$ their images in $H_1(\bar S;\mathbb{F}_2)$. One can verify that they indeed form a symplectic basis under $\overline{(-,-)}$. Passing to the $\mathbb{F}_2$-vector space $V$ introduced in Subsection \ref{subsec:quadratic form & intersection form}, $u_i$ becomes $\{2i-1,2i\}$ and $v_i$ becomes $\{1,2,\cdots,2i-1,n\}$. Let $\mu_i$ be the simple closed curve on $S$ formed by first going along $e_{2i-1}$ then returning along the inverse of $e_{2i}$, or simply write $\mu_i=e_{2i-1}e_{2i}^{-1}$. Similarly, let $\nu_i$ be the simple closed curve on $S$ expressed by $e_1e_2^{-1}\cdots e_{2i-1}e_{n}^{-1}$. Then $\mu_i,\nu_i$ are simple closed curves on $S$ representing the homology classes $u_i,v_i$, respectively. So $\bar u_1,\bar v_1,\cdots,\bar u_g,\bar v_g$ is in fact a symplectic geometric basis.
\begin{lemma}\label{lem:Arf invariant}
    The Arf invariants of $\Lambda$ and $\Lambda'$ are equal.
\end{lemma}
\begin{proof}
    \par Since $\bar q_{\Delta^*}$ and $\bar q_{\Delta'^*}$ are descended from $q_{\Delta^*}$ and $q_{\Delta'^*}$, respectively, we have
    \[\mathop{\mathrm{Arf}}(\Lambda)=\sum_{i=1}^gq_{\Delta^*}(u_i)q_{\Delta^*}(v_i),\quad
    \mathop{\mathrm{Arf}}(\Lambda')=\sum_{i=1}^gq_{\Delta'^*}(u_i)q_{\Delta'^*}(v_i).\]
    Using formulae in Subsection \ref{subsec:quadratic form & intersection form}, one finds that
    \[\begin{aligned}
        &q_{\Delta^*}(u_i)=1,&&q_{\Delta^*}(v_i)=i,&&\forall i\geq1,\\
        &q_{\Delta'^*}(u_i)=1,&& &&\forall i\geq1,\\
        &q_{\Delta'^*}(v_i)=i,&& &&\forall i\geq3,
    \end{aligned}\]
    with exception
    \[q_{\Delta'^*}(v_1)=0,\quad q_{\Delta'^*}(v_2)=1.\]
    But both Arf invariants turn out to be $\frac{g(g+1)}{2}\mod2$.
\end{proof}
\par Now, Theorem \ref{thm:derived equivent} is immediate.
\begin{proof}[Proof of Theorem \ref{thm:derived equivent}]
    For all $n\geq5$, \cite[Theorem 7.4(1)(2)]{APS} holds by Lemma \ref{lem:ribbon surface}, \ref{lem:boundary winding number}. When $n\equiv2\mod4$, \cite[Theorem 7.4(3)(b)(ii)]{APS} holds by the discussion in Subsection \ref{subsec:a familt of closed curves}. When $n$ is odd or $n\equiv0\mod4$, \cite[Theorem 7.4(3)(b)(iii)]{APS} follows from Lemma \ref{lem:Arf invariant}.
\end{proof}

\par Combining Theorem \ref{thm:derived equivent} and \ref{thm:tc-unique}, we obtain the answer to our problem:
\begin{corollary}
    There exist two gentle algebras which are derived equivalent but not tilting-cotilting equivalent.
\end{corollary}

\subsection{Remark on exceptional cases}
\par We note that for $n\leq4$, the gentle algebra $\Lambda$ is derived unique. The cases $n=1,2$ are obvious. The case $n=3$ is contained in \cite[Theorem A]{Bob}. For the case $n=4$, we make it a proposition.
\begin{proposition}
    If $n=4$, the gentle algebra $\Lambda$ is derived unique.
\end{proposition}
\begin{proof}
    Assume that $\Lambda'$ is derived equivalent to $\Lambda$. By Remark \ref{rmk}, we may take $\Lambda'=\Lambda(L,R)$, and let $L=(1,2,3,4)$. If we view $L,R$ as permutations on $E=\{1,2,3,4\}$ as they are written, then the faces of the ribbon graph $\Gamma_{\Lambda'}$ associated with $\Lambda'$ are in bijection with the cycles of $LR$, and the perimeter ($=$ number of boundary segments) of a face equals the length of the corresponding cycle. Since $\Gamma_\Lambda$ has precisely two faces with perimeter 2, so must do $\Gamma_{\Lambda'}$, and $LR$ must decompose as a product of two disjoint transpositions. Among all 4-cycles on $E$ the only one satisfying this condition is $(1234)$. Therefore there are only 4 possibilities for $R$: $R=(1,2,3,4)$, $(2,3,4,1)$, $(3,4,1,2)$ or $(4,1,2,3)$. 
    \par If $R=(2,3,4,1)$ or $(4,1,2,3)$, then $\Lambda'$ is of infinite global dimension since it contains a full-relation oriented cycle. But $\Lambda$ is of finite global dimension, so these two possibilities are thus eliminated.
    \par If $R=(3,4,1,2)$, we compare their numerical invariants as before. Both dissected surfaces are of genus $1$, so we are in case (a) of \cite[Theorem 7.4(3)]{APS}. A calculation similar to that in Lemma \ref{lem:boundary winding number} shows that all boundary winding numbers are $-2$. A calculation similar to the one leading to Equation (\ref{eq:winding number}) shows that for $\Lambda$, all winding numbers of curves $\gamma_{ij}$ ($1\leq i,j\leq 4$) are 0, so we obtain $\gcd(0,0,-2+2,-2+2)=0$. For $\Lambda'$, notice that $\gamma_{13},\gamma_{14}$ are simple closed curves intersecting transversely at only one point, so they form a symplectic geometric basis of $H_1(\bar S;\mathbb{F}_2)$ and we obtain $\gcd(\pm2,\pm2,-2+2,-2+2)=2$. The numerical invariants disagree, so $\Lambda$ and $\Lambda'$ are not derived equivalent.
    \par Therefore, $R=L$ and $\Lambda'\cong\Lambda$. Thus $\Lambda$ is derived unique.
\end{proof}


\bibliographystyle{amsplain}

\end{document}